\documentclass[11pt,reqno]{amsart}

\usepackage{amsmath}
\usepackage{amsthm}
\usepackage[T1]{fontenc}
\usepackage{mathrsfs}
\usepackage{latexsym}
\usepackage{amssymb}
\usepackage{bm} 
\usepackage{mathabx}
\usepackage{graphicx}
\usepackage{float}
\usepackage{extarrows}
\usepackage{epstopdf}
\usepackage{caption}
\usepackage{subcaption}
\usepackage{enumerate}   
\usepackage{hyperref}
\hypersetup{
    colorlinks,
    linkcolor={black},
    citecolor={luh-dark-blue},
    urlcolor={orange}
}
\usepackage{xcolor}
\usepackage[a4paper,left=3.5cm, right=3cm,top=3cm, bottom=3cm]{geometry}
\floatstyle{boxed}
\restylefloat{figure}

\newcommand{\Conv}{\mathop{{\raisebox{-0.2ex}{$\Asterisk$}}}}%

\numberwithin{equation}{section}

\newcommand{\F}{\mathbb{F}}

\newcommand{\R}{\mathbb{R}}

\newcommand{\N}{\mathbb{N}}

\renewcommand{\i}{\mathrm{i}}

\normalsize
\DeclareMathOperator{\sech}{sech}

\newcommand{\ub}{\bar u}
\newcommand{\vb}{\bar v}
\newcommand{\eb}{\bar \eta}

\newcommand{\ut}{\tilde u}
\newcommand{\vt}{\tilde v}
\newcommand{\et}{\tilde \eta}

\newcommand{\la}{\lambda}
\newcommand{\p}{\partial}

\renewcommand{\vec}{\mathbf}

\newtheorem{theorem}{Theorem}[section]
\newtheorem{proposition}[theorem]{Proposition}
\newtheorem{lemma}[theorem]{Lemma}

\newtheorem{remark}[theorem]{Remark}

\newtheorem{Example}[theorem]{Example}

\usepackage{color}
\definecolor{luh-dark-blue}{rgb}{0.0, 0.313, 0.608}
\definecolor{lred}{rgb}{1.0,0.5,0.5}

\title{Symmetric solutions of evolutionary partial differential equations}
\author{Gabriele Bruell, Mats Ehrnstr\"{o}m, Anna Geyer and Long Pei}
\address{Department of Mathematical Sciences, Norwegian University of Science and Technology, 7491 Trondheim, Norway.}
\email{gabriele.bruell@math.ntnu.no}
\email{mats.ehrnstrom@math.ntnu.no}
\address{Institute of Applied Mathematics, Delft University of Technology, 2628 CD Delft
The Netherlands.}
\email{A.Geyer@tudelft.nl}
\address{Department of Mathematics, KTH Royal Institute of Technology in Stockholm, 11428 Stockholm, Sweden.}
\email{longp@kth.se}

\thanks{Part of this research was carried out as A.G. visited Trondheim under the support of grant no. 231668 from the Research Council of Norway; G.B., M.E. and L.P. additionally recognise the support of grant no. 250070 from the same source.}

\begin{document}

\maketitle

\begin{abstract}
We show that for a large class of evolutionary nonlinear and nonlocal partial differential equations, symmetry of solutions implies very restrictive properties of the solutions and symmetry axes. These restrictions are formulated in terms of  three principles, based on the structure of the equations. The first principle covers equations that allow for steady solutions and shows that any spatially symmetric solution is in fact steady with a speed determined by the motion of the axis of symmetry at the initial time. The second principle includes equations that admit breathers and steady waves, and therefore is less strong: it holds that the axes of symmetry are constant in time. The last principle is a mixed case, when the equation contains  terms of the kind from both earlier principles, and there may be different outcomes; for a class of such equations one obtains that a spatially symmetric solution must be constant in both time and space. We list and give examples of more than \(30\) well-known equations and systems in one and several dimensions satisfying these principles; corresponding results for weak formulations of these equations may be attained using the same techniques. Our investigation is a generalisation of a local and one-dimensional version of the first principle from [E., Holden, and Raynaud, 2009] to nonlocal equations, systems and higher dimensions, as well as a study of the standing and mixed cases.
\end{abstract}

\section{Introduction}
\label{S-Intro}

In this article we investigate the consequences of a priori spatial symmetry of solutions to a class of partial differential equations of the general form  
\begin{equation}
\label{E-general equation}
    P(D)u_t=F(D,u),
\end{equation}
where  $u$ is a classsical, sufficiently smooth solution, $P$ is a linear Fourier multiplier operator and $F(D,u)$ is a nonlinear and possibly nonlocal function of $u$ and its derivatives, where we use the notation $D:= -i\partial_x$. Let \(I\) be an interval of existence for the equation \eqref{E-general equation}, typically of the form \([0,T)\). We then call a function  $u=u(t,x)$  \emph{spatially symmetric} if there exists a function $\lambda\in C^{1}(\R)$ such that for every $t \in I$ and $x\in \R$
    \begin{equation*}\label{sym def}
           u(t,x)=u(t,2\lambda(t)-x),
    \end{equation*} 
and we call \emph{$\lambda=\lambda(t)$} the \emph{axis of symmetry}. This is in the one-dimensional scalar case, and we shall later generalise this concept to systems and higher dimensions. 

Symmetries of solutions in partial differential equations have been studied for a long time and it is known that, for instance, rotationally invariant elliptic operators naturally impose symmetry of corresponding solutions. The most famous work in this direction is probably that of Gidas, Ni, and Nirenberg in \cite{GNN79}, which in turn was based on the method of moving planes, introduced by Alexandroff \cite{Alk} and Serrin \cite{S71}. Li  later generalised the results from \cite{GNN79}  to fully nonlinear elliptic equations \cite{Li91b, Li91}, and it has some implications for our results, as described below.

The strong connection between symmetry and steady solutions has been observed also in the context of water waves (first by Garabedian \cite{Garabedian1965} for periodic waves, and by Craig and Sternberg \cite{Craig1988} for solitary waves). This phenomenon in fact is not restricted to the Euler equations, but arises from a structural condition related to ellipticity, or more generally to the parity of the symbol appearing in the equation, see \cite{Ehrnstrom2009a, MR1749871, GNN79}. Indeed, as follows from the result \cite{Li91} by Li and the principle (P1) presented below, for a large set of equations the sets of spatially symmetric and travelling solutions completely coincide; an example of this behaviour is the Whitham equation, see \cite{MR3603270}. Note that this is not the case for the free-boundary Euler equations: although we give in Section~\ref{S-systems} the Euler equations as an example of systems belonging to principle (P1), and even though one can show that large classes of its steady solutions are symmetric \cite{ConEhrWah07,walsh2009}, there are also non-symmetric steady solutions \cite{KL2017a}. 

Our viewpoint is slightly different from that of \cite{GNN79} and related work: we \emph{assume} that one has a symmetric, generally time-dependent, solution of the evolutionary equation, with some given axes of symmetry, and study the consequences of it. As is seen from principles (P1) and (P3), this may be very restrictive, even enforcing zero solutions. Whereas the generic case in principle (P1) is hyperbolic, it is parabolic in (P2), and mixed in (P3). We emphasise that our considerations in this paper are exact, but we do not fix a functional-analytic setting: the proofs could be applied in a periodic or localised setting, on manifolds, or for weak solutions. In each case one should use the conditions given by that setting for uniqueness of the initial-value problem. In some cases, such as localised settings, one could sometimes further restrict the kind of possible solutions. We will generally assume that our solutions are well-defined and unique with respect to initial data on some time interval \(I\).

It will not be specified what \(I\) is, but in general one just needs some open set in time to conclude in the proofs. 

At the core of our results lies the following general, albeit local and one-dimensional, principle which states that for a large class of partial differential equations  the spatially symmetric solutions are a subclass of the steady solutions. The following version is taken from \cite[Thm 2.2]{Ehrnstrom2009a}.

\begin{theorem}[Principle (P1)]
\label{T-P1}
Let $P$ be a polynomial and consider the equation
    \begin{equation}
    \label{E-P1}
        P(\p_{x})u_{t} = F (\partial_x,u),
    \end{equation}
where $F (\partial_x, u) = \bar{F} (u, \p_{x}u, \dots , \p_{x}^{n}u)$ for $ n\in \N$,  and it is assumed that \eqref{E-P1} admits at most one classical solution \(u \colon \R \times I \to \R\) for given initial data $u_0=u(0, \cdot)$. 
    If $P$ is even in $\partial_x$ and $F$ is odd in the sense that
    \begin{equation}
    \label{E-F odd}
        \bar{F̄}(a_{0}, -a_{1}, a_{2},-a_{3},\dots ) = -\bar{F̄}(a_{0}, a_{1}, a_{2},a_{3},\dots),
    \end{equation}
    or if $P$ is odd in $\partial_x$ and $F$ is even in the sense that
    \begin{equation}
    \label{E-F even}
        \bar{F̄}(a_{0}, -a_{1}, a_{2},-a_{3}, \dots  ) = \bar{F}(a_{0}, a_{1}, a_{2},a_{3}, \dots  ),
    \end{equation}
 for all $a_i\in \R$, then any spatially symmetric solution of \eqref{E-P1}  is steady with speed \(\dot \lambda(t_0)\), \(t_0 \in I\).
\end{theorem}

Note that although \(t_0 \in I\) in Theorem~\ref{T-P1} is arbitrary, it is part of the conclusion that \(\dot\lambda\) is in fact constant on the entire interval of existence. It is of interest here that, as proved in \cite{Li91}, any positive classical solution $u$ of the steady nonlinear elliptic equation
\begin{align}\label{Eq G}
	-cu&=G(\partial_x, u), \qquad c > 0,
\end{align}
that satisfies decay properties at infinity is symmetric, provided that $G$ is a second-order, sufficiently smooth elliptic operator that is even in the sense of \eqref{E-F even}. Thus, for any equation of the form \(\partial_t u = F(\partial_x, u)\) which falls into the framework of Theorem \ref{T-P1}, and whose steady equation is of the form \eqref{Eq G} with the above conditions satisfied, the set of symmetric and steady solutions coincide.

One of the purposes of the current investigation is to generalise the above result to nonlocal equations, higher dimensions and systems. As showed in Section~\ref{subsec:1D nonlocal}, if one assumes that the nonlocal, and concurrently nonlinear terms are given as products of Fourier multipliers applied to \(u\), the conditions for $F$ to be odd or even translate to corresponding properties of the operator symbols. A second principle, appearing when $P$ and $F$ are of the same parity, is the following.

\begin{theorem}[Principle (P2)]
\label{T-P2}
Consider an equation of the form
    \begin{equation}
    \label{E-P2}
        P(\partial_x)u_{t} = F(\partial_x, u), 
    \end{equation}
 where $F$, $P$ and $u$ are as in Theorem~\ref{T-P1}, but we now assume either that both $P$ and $F$ are even, in the sense of \eqref{E-F even}, or that both $P$ and \(F\) are odd, in the sense of \eqref{E-F odd}.
  If \eqref{E-P2} admits at most one classical solution \(u \colon \R \times I \to \R\) for given initial data $u_0=u(0, \cdot)$, then any spatially symmetric solution of \eqref{E-P2} has a fixed axis of symmetry. 
\end{theorem}

It is a consequence of Theorem~\ref{T-P2} that a solution of \eqref{E-P2} cannot be symmetric and travelling with a non-zero speed at the same time, since the axis of symmetry does not move (this is only possible if the solution is constant). When \eqref{E-P2} is furthermore parabolic of the form $\partial_t u =F(\partial_x, u)$ with \(F\) strictly elliptic, a result by Li \cite{Li91} shows that a positive classical solution tending to zero at infinity, which is unique with respect to its symmetric initial data having only one local maximum, remains symmetric with exactly one local maximum from which the function decreases monotonically. An example of solutions that fall into principle (P2) are symmetric breathers, appearing for instance in the nonlinear Schr\"odinger equation \cite{NLS-breather}. But there are also other solutions of equations in this category, that do not show temporal or spatial recurrence. These examples show that it is generally not possible to deduce stronger structural properties, such as in principle (P1), for equations of the form \eqref{E-P2}.

Whereas the first two principles concern equations whose right-hand side as a whole has a fixed parity, the next principle involves equations whose right-hand side is mixed in the sense that it involves terms of different parity.

\begin{theorem}[Principle (P3)]
\label{T-P3}
Consider the equation
    \begin{equation*}
    \label{E-P3}
        u_{t} = \left(F_1(u)\right)_x + F_2(\partial_x, u), 
    \end{equation*}
    where \(F_1 \colon \R \to \R\) is a local function with \(F_1^\prime\)  invertible in the range of any solution \(u\), and $F_2(\partial_x, u)=  \bar{F}_2 (u, \p_{x}u, . . . , \p_{x}^{n}u)$ is even in the sense of \eqref{E-F even}. Then any spatially symmetric solution of \eqref{E-P3} depends only on time (and hence is spatially constant). If, in addition, $F_2$ is of the form $F_2(\partial_x, u)=\partial_x G(\partial_x,u)$, the solution is constant also in time.
\end{theorem}

The form of \eqref{E-P3} is more general than that of the corresponding two earlier equation types, but there appears to be fewer name-given equations strictly in this category (not belonging to (P1) or (P2)). To show the generality of these principles, we list below quite a few equations to which they are applicable. In several cases, this holds also for their nonlocal, higher-dimensional, or system generalisations, as exemplified in the later sections. The principles apply to equations including the following (references chosen with brevity as leading precept):\\[-12pt]

\begin{itemize}
    \item[(P1)] In \cite{Lannes13a}, Korteweg--de Vries type, Benjamin--Bona--Mahony, Camassa--Holm,\\ Degasperis--Procesi, Green--Naghdi,  Kadomtsev--Petviashvili, Kawahara,\\ Whitham; in \cite{GS07}, Dym, Rosenau--Hyman, Khokhlov--Zabolotskaya,\\ Hirota--Satsuma, Hunter--Saxton; in \cite{Hirota},
    Kaup--Kuperschmidt, Sawada--Kotera; in \cite{PS11}, two-component Camassa--Holm; in \cite{A09},  Benjamin--Ono; in \cite{GX09}, Geng-Xue;
    and in \cite{O78}, Ostrovsky.\\[-12pt]
    
    \item[(P2)] In \cite{Wu01}, heat, porous medium, fast diffusion, thin film, Cahn--Hilliard;
    In \cite{TPDE3}, Kolmogorov--Petrovsky--Piscounov, Fitzhugh--Nagumo; in \cite{Z10},
    Keller--Segel, and in \cite{Lannes13a}, nonlinear Schr\"{o}dinger.\\[-12pt]
    
    \item[(P3)] In \cite{GS07}, Burgers, Kuramoto--Sivashinsky; in \cite{HH}, Buckmaster; in \cite{SZJ07}, derivative nonlinear Schr\"{o}dinger. 
\end{itemize}

\medskip

In the next section we present the basic proofs of the principles (P1), (P2) and (P3) for local equations, and then generalise these to the nonlocal setting. The latter includes for instance, in the case of principle (P1), the Benjamin-Ono equation and equations of general Whitham type, and, in the case of the principle (P2), the nonlocal Keller--Segel equation. In Section \ref{S-higher dimensions} we extend the principles to higher dimensions (in which one could have one or several axes of symmetry), and in Section~\ref{S-systems} to systems. An example in this last section is the 2D-Euler equations in physical vacuum, for which we show that any horizontally symmetric solution is steady. 
\medskip

Finally, note that all results in this paper are stated for real or complex classical solutions defined on \(I \times \R^n\), but that these can be carried over to appropriate functional-analytic settings, including weak ones, by methods such as in \cite{Ehrnstrom2009a} and \cite{Geyer2016}. In general, such a setting is required for the uniqueness results that the principles rely on.

\section{The one-dimensional case}
\label{S-Local Principles}

We begin this section by proving the principles (P1), (P2) and (P3) in the local setting.

\subsection{Local equations}

Consider a local partial differential equation of the form
\[
	P(\partial_x)u_t = F(\partial_x,u),
\]
where $u$ is a sufficiently smooth function, $P$ is a polynomial in $\partial_x$ with constant coefficients, and $F$ is a nonlinear function of $u$ and its derivatives. If $P$ and $F$ have opposite parity in the sense of \eqref{E-F odd} and \eqref{E-F even}, then principle (P1) formulated in Theorem \ref{T-P1} guarantees that any  spatially symmetric solution is steady, provided that the equation admits a unique solution with respect to given initial data. We give here the proof of that principle, which can also be found in \cite{Ehrnstrom2009a}, because the philosophy behind it is the basis for the more general proofs and principles to come. \emph{Note that in this and all proofs to come, we will suppress the dependence upon \(\partial_x\) or \(D\) in the notation for the nonlinear operators, and simply write \(F(u)\)}. We keep the dependence elsewhere to emphasise the general nonlinear form of \(F(\partial_x,u)\).

\begin{proof}[Proof of Theorem \ref{T-P1} (Principle (P1))]
Assume that $P$ is even and $F$ is odd (the proof of the other case is similar).  Observe
 first that if $\bar u(t,x)=v(x-ct)$ is a
 steady solution, then
 \begin{align*}
   P(\partial_x)\partial_t\bar u=P(\partial_x)(-c\partial_x v(x-ct))=-cP(\partial_x)(\partial_x v)(x-ct)
 \end{align*}
 and
 \begin{align*}
   F(\bar u)=F(v(x-ct))=F(v)(x-ct).
 \end{align*}
Therefore, $v$ satisfies
 \begin{equation}\label{eq:U}
   \big(P(\partial_x) \bar u_t - F(\bar u)\big)(t,x)= \big(- c P(\partial_x) v_x-F(v)\big)(x-ct)=0,
 \end{equation} 
and we find that $\bar u$ is a steady solution if and only if $-cP(\partial_x)(\partial_x v)=F(v)$. Let now $u(t,x) = {u(t,2\lambda(t)-x)}$ be a spatially symmetric solution of $P(\partial_x) u_t = F(u)$. Then 
 \begin{align*}
   0 &= (P(\partial_x) \partial_t - F)(u(t,x))\\ 
   &= (P(\partial_x) \partial_t - F)(u(t,2\lambda(t)-x))\\ 
   &= \left(P(\partial_x) (u_t + 2\dot\lambda u_x) + F(u)\right) \big|_{(t,2\lambda(t)-x)},
 \end{align*} 
 where we have used the oddness and evenness of $F$ and $P(\partial_x)$, respectively. Because $x$ is
 arbitrary, we infer that
 \[
 P(\partial_x) u_t = F(u) = -P(\partial_x) (u_t + 2\dot\lambda u_x),
 \]
 and therefore
 \[
 F(u) = -\dot\lambda P(\partial_x) u_x.
 \]
 Fix a time $t_0$, define $c=\dot\lambda(t_0)$, and introduce the function
 \begin{equation*}
   \bar u(t,x)= u(t_0, x-c(t-t_0)).
 \end{equation*}
The function $\bar u$ defines a steady solution since it satisfies equation \eqref{eq:U}. It also coincides with $u$ at $t = t_0$, that is, $\bar u(t_0,\cdot)=u(t_0,\cdot)$.
 From uniqueness with respect to initial
 data, it follows that $u(t,x) =u(t,2\lambda(t) - x) =
 u(t_0, x-c(t-t_0))=\bar u(t,x)$ for all $t$, and thus \(u\) is steady with speed \(c = \dot\lambda(t_0)\). 
\end{proof}

We turn now to the proof of principle (P2), which in fact is so short that it is almost trivial in this setting. It determines the symmetry axis \(\lambda\) when $P$ and $F$ are of the same parity.

\begin{proof}[Proof of Theorem \ref{T-P2} (Principle (P2))]
Suppose $u_0(x)=u_0(2\lambda_0-x)$ and $u$ is a unique solution of \eqref{E-P2} with respect to initial data. Set $v(t,x):= u(t, 2\lambda_0-x)$. Assume that $P$ and $F$ are even, then
\[
	P(\partial_x)v_t(t,x)=P(\partial_x)u_t(t,2\lambda_0-x)=F(u)(t,2\lambda_0-x)=F(v)(t,x).
\]
Since $v(0,x)=u(0,2\lambda_0-x)=u(0,x)$, the assertion follows by uniqueness of the solution with respect to initial data. The case when both \(P\) and \(F\) are odd   is analogous.
 \end{proof}

Whereas the first two principles treat equations whose left- and right-hand sides either have opposite or identical parity, we consider in Theorem \ref{T-P3} a type of equation whose right-hand side admits a mix of even and odd terms.

\begin{proof}[Proof of Theorem \ref{T-P3} (Principle (P3))]
Let $\la\in C^1(\R)$ and assume that $u(t,x)=u(t,2\la(t)-x)$ solves \eqref{E-P3} for all $(t,x)\in I \times\R$. Then 
\begin{align*}
\p_{t}u(t,x)&=u_{t}(t,2\lambda(t)-x)+2\dot{\lambda}(t)u_{x}(t,2\lambda(t)-x), \\
\partial_x (F_1(u))(t,x) &=  - (F_1^\prime(u) u_x) (t,2 \lambda(t) - x)
\end{align*}
and 
\begin{equation*}
    F_2(u)(t,x)=    F_2(u)(t,2\la(t)-x),
\end{equation*}
where the second equality follows from the oddness of \(\partial_x\) (recall that \(F_1\) is a function \(\R \to \R\), whence even in the sense of \eqref{E-F even}), and the last from the evenness of $F_2$. Therefore, 
\begin{equation}
  \label{E-P3 sym var}
        \big(u_{t}+2\dot{\lambda}(t)u_{x}\big)(t,2\lambda(t)-x)  =- (F_1^\prime(u) u_x) (t,2 \lambda(t) - x) + F_2(u)(t,2\lambda(t)-x).
\end{equation}
Since \eqref{E-P3 sym var} holds for all $(t,x)\in I \times \R$ we may evaluate both  \eqref{E-P3 sym var} and  \eqref{E-P3} at $(t,x)$, subtract the two equations and obtain that 
\begin{equation*}
\label{relation P3}
    \left(\dot \la + F_1^\prime(u) \right) u_x=0.
\end{equation*}
Since \(F_1^\prime\) is invertible in the range of \(u\), both $u_x=0$ and $\dot \la(t) + F_1^\prime(u) =0$ imply that $u$ is independent of $x$. If in addition $F_2(u)=\partial_x G(u)$, then the structure of \eqref{E-P3} implies that $u$ is constant in time, too. 
\end{proof}

\begin{remark}
Note that principle (P3) is based on the assumption of nonlinearity. When \(F_1\) is linear, then \(F_1^\prime\) is never invertible. As the example  $u(t,x)=\cos(x+t)$ being a solution of $\partial_t u = \partial_x u$ shows, the invertibility assumption cannot in general be made away with.
\end{remark}

\begin{remark}\label{rem:local p3 general}
The more general result, when considering an equation of the form
    \begin{equation}
    \label{E-P3R}
        u_{t} =F_1(\partial_x,u) + F_2(\partial_x,u),  
    \end{equation}
where $F_1$ and  \(F_2\) are odd and even operators in the sense of \eqref{E-F odd} and \eqref{E-F even}, respectively, is that, at any instant in time, a spatially symmetric solution of \eqref{E-P3R} is a steady solution of the equation $- \dot \lambda(t) u_x=F_1(\partial_x,u)$. 

\begin{proof}
Using the symmetry of \(u\) and the fact that $F_1$ is odd while $F_2$ is even, one obtains that $u$ also solves
    \[
	    u_t +2\dot{\lambda} u_x=-F_1(u)+F_2(u),
    \]
and  by subtracting this equation from \eqref{E-P3R}, one finds 
    \begin{equation*}
    \label{E1}
	    -\dot{\lambda}u_x=F_1(u).
    \end{equation*}
    \end{proof}
\noindent Therefore, \([t_1,t_2] \ni t \mapsto u(t,\cdot)\) is an orbit between steady states of \(-c \varphi_x=F_1(\partial_x,\varphi)\). Homoclinic orbits contain time-periodic solutions. As the following example shows, unless \(F_2\) is tailored to \(F_1\), uniqueness of even, non-trivial, steady states of the equation \(-c \varphi_x=F_1(\partial_x, \varphi)\) will guarantee that \(u\) is in fact trivial.    

\begin{Example}
 An example of \eqref{E-P3R} is the (viscous) KdV--Burgers equation,
    \begin{equation}\label{KdVB}
	    u_t = 6uu_x- u_{xxx}+\nu u_{xx}, \qquad \nu > 0,
    \end{equation}
with  $F_1(u)=6uu_x-u_{xxx}$, \(F_2(u) = \nu u_{xx}\). Using the symmetry assumption on \(u\) we obtain     \[
    -\dot{\lambda}u_x =6uu_x-u_{xxx}.
    \]
    Hence $u(t,\cdot)$ satsifies the steady KdV equation with respect to the wave speed $c=\dot{\lambda}(t)$ at any instant of time. In an \(L^2(\R)\)-setting this means
    \[
	    u(t,x)=\frac{1}{2}\dot{\lambda}(t)\sech^2 \left({\frac{1}{2}} \sqrt{\dot{\lambda}(t)}(x - \dot\lambda(t)t)\right),
    \]
    unless \(u(t;\cdot) =0\). Plugging $u$ into \eqref{KdVB}, one deduces after some computations that $\dot{\lambda}(t)=0$ for all $t\in I$, whence $u=0$ is the only possible classical solution in \(L^2(\R)\) of the KdV--Burgers equation which is symmetric at any instant of time. 
    \end{Example}
    
\end{remark}

\subsection{Nonlocal equations}\label{subsec:1D nonlocal}
Assuming that nonlocal, nonlinear, terms may be expressed using (several) Fourier multipliers, we may apply similar ideas as in the previous section to prove the respective principles for nonlocal equations. This is achieved by considering the equations on the Fourier side. The Fourier transform, of course, is well defined for tempered distributions, so this approach is essentially not less general than the methods applied for local solutions. 

Of concern are nonlinear, nonlocal equations of the form
\begin{equation*}
     P(D) u_t = F(D,u),
\end{equation*}
where $P$ is a (linear) Fourier multiplier operator and  $F(D,u)=\bar F(u,K_1,\ldots, K_m)$ is a nonlinear and nonlocal function of $u$,  and Fourier multiplier operators $K_i, i\in \{1,\ldots, m\}$; any derivatives of \(u\) may of course be included in the Fourier multipliers \(K_i\). Here, \(D = -\i \partial_x\), and on the Fourier side we study the equation
\begin{equation*}
    P(\xi) \hat{u}_t(t,\xi) = \F(\hat{u}(t,\cdot),\xi),
\end{equation*}
where $\hat{u}$ denotes the Fourier transform of $u$ with respect to the space variable and $\F(\hat{u}(t,\cdot),\xi) = \mathcal{F}(F(D,u))(t,\xi)$ and $P(\xi)$ is the symbol of $P(D)$. Note that the Fourier transform of $F$ is of the form
\begin{equation}\label{eq:Fourier form}
\F(\hat{u}(t,\cdot),\xi)=\sum_{k=1}^l h_k(\xi)\left[\Conv_{j=1}^{n_k} g_{j,k}(\cdot)\hat{u}(t,\cdot)\right](\xi),
\end{equation}
for some $l\in\N$, where $h_k$ and $g_{j,k}$, $j\in \{1,\ldots,n_k\}$, are Fourier multipliers and we define
\[
\Conv_{j=1}^n f_j = f_1 *\cdots*f_n,
\]
where for  $n=1$ we adopt the convention that $\Conv_{j=1}^1 f_j =f_1$.
Note that the second argument of \(\F\) denotes the variable \(\xi\) as it appears in the symbols \(h_k\) and \(g_{j,k}\). Therefore,
\begin{equation}\label{eq:notationF}
\F(\hat{u}(t,\cdot),-\xi) =\sum_{k=1}^l h_k(-\xi)\left[\Conv_{j=1}^{n_k} g_{j,k}(-\cdot)\hat{u}(t,\cdot)\right](\xi),
\end{equation}
and this is in general not the same as \(\F(\hat u(t,\cdot), \cdot)|_{-\xi}\). Because \(\widehat{\partial_x} = \i \xi\), the condition that \(F\) be odd or even translates to that  
\[
\xi \mapsto \F(\hat{u}(t,\cdot),\xi)
\]
be odd or even in $\xi$, respectively. Similarly, \(P\) even/odd means just that the symbol \(P(\xi)\) is even/odd in \(\xi\). And a solution $u$ is spatially symmetric around  $\lambda\in C^1(\R)$ exactly if 
\begin{equation}
\label{symmetric Fourier}
\hat{u}(t,\xi)=e^{-\i 2\lambda(t)\xi}\hat{u}(t,-\xi),
\end{equation}
for all $(t,\xi)\in I \times \R$.

\begin{theorem}[Principle (P1) for nonlocal equations]
\label{P1 non-local}
Consider the equation
\begin{equation}
\label{eq: nonlocal}
    P(D) u_t = F(D,u),
\end{equation}
where  $P(D)$ is a Fourier multiplier operator and the nonlinear, nonlocal operator $F(D,\cdot)$ is a pseudo-product of the form \eqref{eq:Fourier form}, and we assume that the equation has a unique solution $u=u(t,x)$ for each given initial datum $u_0=u(0,\cdot)$. If $\xi \mapsto P(\xi)$ is even and \(\xi \mapsto \F(\hat{u}(t,\cdot),\xi)\) is odd, or if $P$ is odd and $\F$ is even, then any spatially symmetric solution of \eqref{eq: nonlocal} is steady. 
\end{theorem}

The following lemma is the key ingredient in proving Theorem \ref{P1 non-local}. 

\begin{lemma}\label{main equality lemma} Under the assumptions of Theorem \ref{P1 non-local}, let $u$ be spatially symmetric around $\lambda\in C^1( \R)$. Then
\begin{equation}
\label{main equality}
    \F(\hat{u}(t,\cdot),\xi) = \pm e^{-\i 2\lambda(t)\xi} \F(\hat{u}(t,\cdot),\cdot)|_{-\xi},
\end{equation}
where the sign is positive/negative according to whether \(\xi \mapsto \F(\hat{u}(t,\cdot),\xi)\) is even/odd, respectively.
\end{lemma}

\begin{proof}
The Fourier transformation of $F(u)$ is
\[\F(\hat{u}(t,\cdot),\xi)=\sum_{k=1}^l h_k(\xi)\left[\Conv_{j=1}^{n_k} g_{j,k}(\cdot)\hat{u}(t,\cdot)\right](\xi).\]
Considering each term separately, we show by induction that
\begin{equation}\label{indu}
\left[\Conv_{j=1}^{n} g_{j}(\cdot)\hat{u}(t,\cdot)\right] \bigg|_{-\xi} =  \left[\Conv_{j=1}^{n} g_{j}(-\cdot)\hat{u}(t,-\cdot)\right]\bigg|_{\xi}.
\end{equation}
Clearly, this holds for $n=1$.  Assuming that \eqref{indu} holds $n-1$, one then obtains that
\begin{align*}
\left[\Conv_{j=1}^{n} g_{j}(\cdot)\hat{u}(t,\cdot)\right]\bigg|_{-\xi} &= \int_{-\infty}^\infty\left[\Conv_{j=1}^{n-1} g_{j}(\cdot)\hat{u}(t,\cdot)\right](y)g_n(-\xi-y)\hat{u}(t,-\xi-y) \, dy\\
&= \int^{-\infty}_\infty \left[\Conv_{j=1}^{n-1} g_{j}(\cdot)\hat{u}(t,\cdot)\right](-z)g_n(-\xi+z)\hat{u}(t,-\xi+z) \, dz \\
&=\int_{-\infty}^\infty\left[\Conv_{j=1}^{n-1} g_{j}(-\cdot)\hat{u}(t,-\cdot)\right](z)g_n(-(\xi-z))\hat{u}(t,-(\xi-z)) \, dz\\
&= \left[\Conv_{j=1}^{n} g_{j}(-\cdot)\hat{u}(t,-\cdot)\right]\bigg|_{\xi},
\end{align*}
where we used the substitution $y=-z$ in the second line, and the induction assumption in the third. From \eqref{indu} one gets to \eqref{main equality} as
\begin{align*}
\label{equalities}
\begin{split}
\F(\hat{u}(t,\cdot),\cdot)|_{-\xi} &= \sum_{k=1}^l h_k(-\xi)\left[\Conv_{j=1}^{n_k} g_{j,k}(-\cdot)\hat{u}(t,-\cdot)\right]\bigg|_{\xi} \\
&= \sum_{k=1}^l h_k(-\xi)\left[\Conv_{j=1}^{n_k} g_{j,k}(-\cdot)e^{\i 2\lambda(t)(\cdot)}\hat{u}(t,\cdot)\right]\bigg|_{\xi} \\
&= e^{\i 2\lambda(t)\xi} \F(\hat{u}(t,\cdot),-\xi)\\
&= -e^{\i 2\lambda(t)\xi} \F(\hat{u}(t,\cdot),\xi),
\end{split}
\end{align*}
where we have used the induction argument \eqref{indu} in the first line, the symmetry \eqref{symmetric Fourier} of \(u\) in the second, the form \eqref{eq:notationF} of \(\F\) in the third, and the oddness of \(\xi \mapsto \F(\hat u(t,\cdot), \xi)\) in the last. The corresponding argument when \(\xi \mapsto \F(\hat u(t,\cdot), \xi)\) is even yields instead a positive sign in \eqref{main equality}.
\end{proof}

Following the idea of the proof of Theorem~\ref{T-P1}, we are now ready to prove the analogous result for the nonlocal equation \eqref{eq: nonlocal}.

\begin{proof}[Proof of Theorem \ref{P1 non-local}]
We give the proof for the case when $P$ is even and $F$ is odd; the opposite case is analogous. Suppose that $u$ is a solution of \eqref{eq: nonlocal}, so that
\begin{equation}
\label{Fourier side}
    P(\xi) \hat{u}_t(t,\xi) = \F(\hat{u}(t,\xi),\xi).
\end{equation}
Assuming that \(u\) is spatially symmetric, we obtain from \eqref{symmetric Fourier} that $\hat u$ satisfies
\begin{align}
\label{u_t symm}
\p_t \hat{u}(t,\xi) &=e^{-\i 2\lambda(t)\xi}\hat{u}_t(t,-\xi)-\i 2\dot{\lambda}(t)\xi e^{-\i 2\lambda(t)\xi}\hat{u}(t,-\xi),
\end{align}
and from Lemma~\ref{main equality lemma} that
\begin{align}
\label{G symm}
 \F(\hat{u}(t,\cdot),\xi) = -e^{-\i 2\lambda(t)\xi} \F(\hat{u}(t,\cdot),\cdot)|_{-\xi}.
\end{align}
Since $\hat{u}$ solves \eqref{Fourier side}, the equalities \eqref{u_t symm} and \eqref{G symm} imply
\begin{equation}
\label{equation in symm}
      P( \xi)\left ( \hat{u}_t(t,-\xi)-\i 2\dot{\lambda}(t)\xi \hat{u}(t,-\xi) \right )= -\F(\hat{u}(t,\cdot),\cdot)|_{-\xi}.
\end{equation}
Now, evaluating \eqref{Fourier side} at $(t,-\xi)$ gives
\begin{equation*}
\label{equation at -xi}
        P( \xi) \hat{u}_t(t,-\xi) =  \F(\hat{u}(t,\cdot),\cdot)|_{-\xi},
\end{equation*}
and by subtracting  \eqref{equation in symm} one obtains
\begin{equation}
\label{sum}
\i\dot{\lambda}(t)\xi   P( \xi)\hat{u}(t,-\xi)=\F(\hat{u}(t,\cdot),\cdot)|_{-\xi}.
\end{equation}
Fix $t_0 \in I$ and let $c:=\dot{\lambda}(t_0)$. With
$\bar{u}(\cdot):=u(t_0,\cdot)$, its Fourier transform $\hat{\bar{u}}$ satisfies \eqref{sum}:
\begin{equation}
\label{sum t0}
    \i c\xi   P( \xi)\hat{\bar{u}}(-\xi)=\F(\hat{\bar{u}}(\cdot),\cdot)|_{-\xi}.
\end{equation}
As in the local proof one defines then a steady solution
\[
\tilde{u}(t,x):=\bar{u}(x-c(t-t_0)), \qquad \hat{\tilde{u}}(t,\xi)=\hat{\bar{u}}(\xi) e^{-\i c(t-t_0)\xi}.
\]
An easy induction argument, similar to that in the proof of Lemma~\ref{main equality lemma}, shows that
\[\F(\hat{\tilde{u}}(t,\xi),\xi) =\F(\hat{\bar{u}}(\xi),\xi)e^{-\i c(t-t_0)\xi}.\]
Hence,
\begin{align*}
      P( \xi) \frac{d}{dt} \hat{\tilde{u}}(t,\xi) - \F(\hat{\tilde{u}}(t,\xi),\xi)
        &= -\i c\xi  P( \xi)\hat{\bar{u}}(\xi) e^{-\i c(t-t_0)\xi} -\F(\hat{\bar{u}}(\xi),\xi)e^{-\i c(t-t_0)\xi}\\
        &=e^{-\i c(t-t_0)\xi}\left[-\i c\xi   P( \xi)\hat{\bar{u}}(\xi)-\F(\hat{\bar{u}}(\xi),\xi)\right].
\end{align*}
Evaluating \eqref{sum t0} at $\xi$ yields that this expression vanishes, so that $\tilde{u}$ is a solution of \eqref{eq: nonlocal}.
By construction, $\tilde{u}(t_0,\cdot) = \bar{u}(\cdot)= u(t_0,\cdot)$, so that  $\tilde{u}$ coincides with $u$ at $t_0$, and whence, by uniqueness, for all \(t \geq t_0\) in the interval of existence. Thus \(u\) is steady.
\end{proof}

\begin{Example}[Boussinesq--Whitham type equations]
An equation of the form
\begin{equation}
\label{Whitham type}
u_t +\left(n(u)+Lu\right)_x=0,
\end{equation}
where $n \colon \R \to \R$ is a local function, and the nonlocal operator $L$ is a Fourier multiplier operator with an even and real symbol \(m(\xi)\), includes the case of the Whitham, the Benjamin--Ono, and similar behaving equations, cf. \cite{HKNS06}. Rewriting \eqref{Whitham type} in accordance with the notation of Theorem \ref{P1 non-local}, the equation reads
\[u_t =F(D,u)\]
with $F(D,u) = -n^\prime(u)u_x-Lu_x$. The Fourier transform of $F(D,u)$ is 
\[\F(\hat{u}(t,\cdot),\xi)= - \i[\mathcal{F}(n^\prime(u))(t,\cdot)*[\cdot \, \hat{u}(t,\cdot)]](\xi)- \i\xi m(\xi) \hat{u}(t,\xi)\]
and since $m$ is even,
\begin{align*}
\F(\hat{u}(t,\cdot),-\xi) &=- \i[\mathcal{F}(n^\prime(u))(t,\cdot)* [-\cdot \hat{u}(t,\cdot)]](\xi) - \i (-\xi) m(-\xi) \hat{u}(t,\xi) \\
&=-\F(\hat{u}(t,\xi),\xi),
\end{align*}
in the notation of \eqref{eq:notationF}.
Hence, by Theorem \ref{P1 non-local}, any spatially symmetric classical solution of \eqref{Whitham type} which is unique with respect to the initial data is a travelling wave. 
\end{Example}

Just as principle (P1), the principles (P2) and (P3) can be generalised to the nonlocal setting as well.

\begin{theorem}[Principle (P2) for nonlocal equations]
\label{T-P2 nonlocal}
Consider the equation
\begin{equation}\label{equation for P2}
    P(D) u_t = F(D,u),
\end{equation}
where  $P(D)$ is a Fourier multiplier operator and $F(D,\cdot)$ is a pseudo-product of the form \eqref{eq:Fourier form}, and we assume that the equation has a unique solution $u=u(t,x)$ for each given initial datum $u_0=u(0,\cdot)$. If $\xi \mapsto P(\xi)$ and \(\xi \mapsto \F(\hat{u}(t,\cdot),\xi)\) have the same parity, then any spatially symmetric solution of \eqref{equation for P2} has a fixed axis of symmetry.
\end{theorem}

\begin{proof}
Assume without loss of generality that both $P$ and $F$ have even symbols. Let $v(t,x):=u(t,2\lambda_0-x)$. We first verify that $F(u)=F(v)$, meaning $\F(\hat u(t, \cdot), \xi)=\F(\hat v(t,\cdot), \xi)$. On the Fourier side, one has $\hat v(t,\xi)=e^{-\i 2\lambda_0\xi}\hat u(t,2\lambda_0-\xi)$ and
\begin{align*}
\F(\hat v (t, \cdot), \xi)&= \sum_{k=1}^l h_k(\xi)\left[\Conv_{j=1}^{n_k} g_{j,k}(\cdot)\hat{v}(t,\cdot)\right]\bigg|_{\xi} \\
&=\sum_{k=1}^l h_k(\xi)\left[\Conv_{j=1}^{n_k} g_{j,k}(\cdot)e^{-\i 2\lambda_0\xi}\hat{u}(t,-\cdot)\right]\bigg|_{\xi}\\
&=e^{-\i 2\lambda_0\xi}\sum_{k=1}^l h_k(\xi)\left[\Conv_{j=1}^{n_k} g_{j,k}(\cdot)\hat{u}(t,-\cdot)\right]\bigg|_{\xi}\\
&=e^{-\i 2\lambda_0\xi}\sum_{k=1}^l h_k(\xi)\left[\Conv_{j=1}^{n_k} g_{j,k}(-\cdot)\hat{u}(t,\cdot)\right]\bigg|_{-\xi}\\
&= e^{-\i 2\lambda_0\xi}\F(\hat u(t,\cdot),\cdot)|_{-\xi}\\ 
&= \F(\hat u(t,\cdot), \xi),
\end{align*}
where the second to last equality is due to \eqref{indu}, and the last follows from the evenness of $\xi \mapsto \F(\hat u(t,\cdot),\xi)$ and Lemma~\ref{main equality lemma}. A similar, but simpler, argument shows that  $P(D)u_t=P(D)v_t$. We deduce that $v$ solves the same equation as $u$, and the assertion follows from uniqueness of the initial-value problem.
\end{proof}

We finally state the nonlocal version of principle (P3). Its proof is similar to the proof of Theorem~\ref{T-P3} using \eqref{main equality}, and we leave out the details.

\begin{theorem}[Principle (P3) for nonlocal equations]
\label{T-P3 nonlocal}
Consider the equation
    \begin{equation}
    \label{E-P3 nonlocal}
        u_{t} = \left(F_1(u)\right)_x + F_2(D,u), 
    \end{equation}
    where \(F_1 \colon \R \to \R\) is local with \(F_1^\prime\) invertible in the range of any solution \(u\), and $F_2(D,\cdot)$ is a pseudo-product of the form \eqref{eq:Fourier form} with \(\F(\hat u(t,\cdot),\xi)\) even in \(\xi\).  
    Then any spatially symmetric solution of \eqref{E-P3 nonlocal} depends only on time. If, in addition, $F_2$ is of the form $F_2(D,u)=\partial_x G(D,u)$, then the solution is constant also in time.
\end{theorem}

\begin{remark}
The more general comment made in Remark~\ref{rem:local p3 general} holds true also in the case when both \(F_1\) and \(F_2\) are general nonlocal operators of the form \eqref{eq:Fourier form} with \(\xi \mapsto \F_1(\hat u(t,\cdot), \xi)\) odd and \(\xi \mapsto \F_2(\hat u(t,\cdot), \xi)\) even:  any spatially symmetric solution of         \(u_{t} =F_1(D,u) + F_2(D,u)\)  is, at each instant of time, a steady solution of the equation $- \dot \lambda u_x=F_1(D,u)$. 
\end{remark}

\section{Higher dimensions}
\label{S-higher dimensions}

As briefly mentioned in \cite{Ehrnstrom2009a}, the principle (P1) presented in Theorem \ref{T-P1} may be generalised to a multi-dimensional setting. In this section we give a rigorous proof of this result, and  state generalisations of all established principles to higher dimensions. Notice that in higher dimensions the symmetry condition on a solution $u=u(t,\vec x)$ does not necessarily have to hold in each component of $\vec x$. Generally, conclusions can be made only about the components for which one assumes spatial symmetry. The setting is as follows: we write \({\vec x} = ({\vec x}', {\vec x}'')\) and \(\boldsymbol{\xi} = (\boldsymbol{\xi}', \boldsymbol{\xi}'')\) as in
\begin{align*}
{\vec x} &= ({\vec x}', {\vec x}'') = (x_1, \ldots, x_q, x_{q+1}, \ldots, x_p) \in \R^q \times \R^{p-q},
\end{align*}
and consider \({\vec x}'\)-spatially symmetric solutions \(u\) such that
\begin{equation}\label{eq:x'-symmetric}
u(t,\vec{x}',\vec{x}'')=u(t,2\boldsymbol{\lambda}(t)-\vec{x}', \vec{x}''),
\end{equation}
for some axes of symmetry $\boldsymbol{\lambda} =(\lambda_1,\dots,\lambda_q) \in C^1(I,\R^q)$. We consider nonlinear, nonlocal operators \(F\) that are expressible as a sum of products of Fourier multipliers acting on \(u\); such are all pseudo-products, and may with the help of multivariate symbols be expressed as sums of convolutions on the Fourier side. In analogy with the one-dimensional case, an operator \(F\) with \({\mathcal F}(F({\boldsymbol D},u))({\boldsymbol \xi}) = \F(\hat u(t,\cdot), {\boldsymbol \xi})\) will be called even or odd according to whether the function
\begin{equation}\label{eq:higher-dim Fourier parity}
{\boldsymbol \xi}' \mapsto \F(\hat u(t, \cdot), ({\boldsymbol \xi}',{\boldsymbol \xi}''))
\end{equation}
is even or odd (so that the total symbol of \(F\) is even or odd in \({\boldsymbol \xi}'\)). In the case of a local operator \(F(\partial_{\vec x},u)\)  depending only on $u$ and its derivatives, this parity coincides with that of 
 \begin{equation}
\label{mE-F x}
    {\vec x}' \mapsto F (\partial_{\vec x},u(t,\cdot))({\vec x}', {\vec x}'').
\end{equation}
Note that we choose this shorter way of expressing the oddness/evenness assumption presented first in Theorem~\ref{T-P1}, since the conditions otherwise become quite cumbersome to express in the multidimensional setting. We believe, however, that the former is more informative in the original setting.

\begin{theorem}[Principle (P1) for  higher dimensions]
\label{mT-P1}
Consider the equation
    \begin{equation}
    \label{mE-P1}
        P({\boldsymbol D})u_{t} = F ({\boldsymbol D},u), 
    \end{equation}
where $P({\boldsymbol D})$ is a multivariate Fourier multiplier operator, and \(F({\boldsymbol D},\cdot)\) is a multivariate pseudo-product that may be expressed on the Fourier side via \({\mathcal F}(F({\boldsymbol D},u))({\boldsymbol \xi}) = \F(\hat u(t,\cdot), {\boldsymbol \xi})\). Assume that \eqref{mT-P1} admits at most one classical solution $u=u(t, \vec x)$ for given initial data $u_0=u(0, \cdot)$, and let $u$ be an \({\vec x}'\)-symmetric solution in the sense of \eqref{eq:x'-symmetric}. If the functions \({\boldsymbol \xi}' \mapsto \F(\hat u(t, \cdot), ({\boldsymbol \xi}',{\boldsymbol \xi}''))\)  and \({\boldsymbol \xi}' \mapsto P({\boldsymbol \xi}',{\boldsymbol \xi}'')\) are of different parity, then $u$ is steady in the \({\vec x}'\)-direction: there exists a constant vector ${\vec c} \in \R^q$ such that
\[
u(t, {\vec x})= u_0({\vec x}' - {\vec c} t, {\vec x}'').
\]
\end{theorem}

Theorem~\ref{mT-P1} may be proved using the same arguments as in the proof of Theorems~\ref{T-P1} and~\ref{P1 non-local}. We provide here the main steps in the local case \(F(\partial_{\vec x},u)\) to convince the reader that (i) one may introduce an \({\vec x}'\)-steady solution even when there are several axes of symmetry, and (ii) that the \({\vec x}''\)-direction does not affect the proof in any significant way.\footnote{It is easy to come up with a pseudo-product mixing directions in a nonlocal way, such that the separation between the \({\vec x}'\)- and \({\vec x}''\)-directions would be violated. Such operators, however, are excluded by our parity assumptions.}

\begin{proof}[Proof of Theorem \ref{mT-P1}, local case]

Let $u=u(t,{\vec x})$ be a classical solution of \eqref{mE-P1}, which is symmetric in the ${\vec x}'$-variable in the sense of \eqref{eq:x'-symmetric}.
In multi-index notation, we have
\begin{align*}
\label{2D_relations}
\begin{split}
&\p_t u(t,{\vec x}', {\vec x}'') = u_t (t,2 {\boldsymbol \lambda}(t)- {\vec x}', {\vec x}'')+2\dot{\boldsymbol{\lambda}}(t)\nabla_{{\vec x}'} u(t,2\boldsymbol{\lambda}(t)- {\vec x}', {\vec x}''),\\
&\p_{{\vec x}'}^{\boldsymbol \alpha} u(t, {\vec x}', {\vec x}'')= (-1)^{|{\boldsymbol \alpha}|}  ( \p_{{\vec x}'}^{\boldsymbol \alpha} u)(t,2\boldsymbol{\lambda}(t)- {\vec x}', {\vec x}''),\\
&\p_{{\vec x}''}^{\boldsymbol \alpha} u(t, {\vec x}', {\vec x}'')=   ( \p_{{\vec x}''}^{\boldsymbol \alpha} u)(t,2\boldsymbol{\lambda}(t)- {\vec x}', {\vec x}''),
\end{split}
\end{align*}
which very clearly summarises the consequences of the \({\vec x}'\)-symmetry assumption. Assuming that the symbol \(P({\boldsymbol \xi}', {\boldsymbol \xi}'')\) is even in \({\boldsymbol \xi}'\), and that \(F\) is odd in the sense of \eqref{mE-F x}, one also obtains that 
\begin{equation}
\label{sv1}
P(\nabla_{\vec x}) \p_t u(t,{\vec x}', {\vec x}'') = P(\nabla_{\vec x}) \big( u_t (t,2 {\boldsymbol \lambda}(t)- {\vec x}', {\vec x}'')+2\dot{\boldsymbol{\lambda}}(t)\cdot \nabla_{{\vec x}'} u(t,2\boldsymbol{\lambda}(t)- {\vec x}', {\vec x}'') \big),
\end{equation}
and
\begin{equation}
\label{sv2}
F(u)(t,{\vec x}', {\vec x}'') =-F(u)(t,2\dot{\boldsymbol{\lambda}}(t)- {\vec x}', {\vec x}'').
\end{equation}
Since \eqref{sv1} and \eqref{sv2} hold for all $(t,\vec x)\in I\times\R^p$, and \(u\) is a solution of \eqref{mE-P1}, we deduce that
\[
P(\nabla_{\vec x})( u_t (t,\vec x)+2\dot{\boldsymbol{\lambda}}(t) \cdot \nabla_{\vec x'} u(t,\vec{x}))  = -F(u)(t,\vec x).
\]
Subtracting this equation from \eqref{mE-P1} yields
\[
-\dot{\boldsymbol{\lambda}}(t) \cdot \big( P(\nabla_{\vec x})\nabla_{\vec x'} u(t,\vec x) \big) =F(u)(t,\vec x),
\]
valid for all $(t,\vec x)\in I\times\R^p$. From here the proof follows the lines of the earlier proofs. Fixing  a time $t_0>0$ and setting $\vec c:=\dot{\boldsymbol {\lambda}}(t_0)$, the function
\[
    \bar{u}(t,\vec x):=u(t_0,\vec x'-\vec c(t-t_0),\vec x'')
 \]
is an \(\vec x'\)-steady solution of \eqref{mE-P1}, coinciding with the \(\vec x'\)-symmetric solution $u$ at $(t_0,\vec x)$. By uniqueness of the initial-value problem, $u=\bar{u}$.
\end{proof}

The multi-dimensional generalisation of principle (P2) is now straight-forward, using the one-dimensional proof in combination with the notation from the proof of Theorem~\ref{mT-P1}.

\begin{theorem}[Principle (P2) for higher dimensions]
\label{T-P2_2D}
If, in Theorem~\ref{mT-P1}, the functions \({\boldsymbol \xi}' \mapsto \F(\hat u(t, \cdot), ({\boldsymbol \xi}',{\boldsymbol \xi}''))\)  and \({\boldsymbol \xi}' \mapsto P({\boldsymbol \xi}',{\boldsymbol \xi}'')\) are instead of the same parity, then the vectorial axis of symmetry is fixed: there exists a vector ${\boldsymbol \lambda_0} \in \R^q$ such that \({\vec x}' \mapsto u(t,\vec x', \vec x'')\) is symmetric around \({\boldsymbol \lambda}_0\) for all \(t \in I\).
\end{theorem}

Principle (P3) does not have such a natural generalisation to the higher-dimensional setting, but the general comment made in Remark~\ref{rem:local p3 general} still holds, and we state it here for completeness.

\begin{theorem}[Weak principle (P3) for higher dimensions]
\label{T-P3_2D}
Consider the equation
    \begin{equation}\label{eq:P3 general multidim}
        u_{t} =F_1(\boldsymbol D,u) + F_2(\boldsymbol D,u),  
    \end{equation}
where $F_1$ and  \(F_2\) are both multi-variate pseudo-products such as in Theorem~\ref{mT-P1}, with \({\boldsymbol \xi}' \mapsto \F_1(\hat u(t, \cdot), ({\boldsymbol \xi}',{\boldsymbol \xi}''))\) odd and \({\boldsymbol \xi}' \mapsto \F_2(\hat u(t, \cdot), ({\boldsymbol \xi}',{\boldsymbol \xi}''))\) even.  Then,  any \(\vec x'\)-symmetric solution of \eqref{eq:P3 general multidim} is, at each instant of time, a solution of the \(\vec x'\)-steady equation $- \dot {\boldsymbol \lambda} \cdot \nabla_{\vec x'} u =F_1({\boldsymbol D},u)$. 
\end{theorem}

\section{Vector-valued equations}
\label{S-systems}

The above principles may be extended to systems of equations as well, simply by considering each component (where each component \(\vec u_i\) of the solution should be symmetric with respect to the same axes of symmetry \(\boldsymbol \lambda\)). Instead of stating the vector-valued equivalents of the principles from last section, which are just the same with \(\vec u\), \(\vec P\) and \(\vec F\) now boldfaced, we give a couple of examples to show the applicability of the respective principles. Notice that the parities only need to match in each component: in principle (P1), $\vec P$ does not need a parity; only the parity of \(\vec P_i\) and $\vec F_i$ need to match.

\begin{Example}[Hirota-Satsuma equation, (P1)]
A simple example to which the vector-valued principle (P1) applies is the Hirota--Satsuma equation \cite{GS07},
given by
\begin{align*}
    &u_t = \frac{1}{2} u_{xxx} + 3 u u_x - 6ww_x,\\
    &w_t = -w_{xxx}-3uw_x.
\end{align*}
Here, ${\vec P} = (1,1)$ is even, while ${\vec F}$ is odd in \(\partial_x\), and we conclude that any spatially symmetric and unique classical solution $(u,w)$ is steady.
\end{Example}

\begin{Example}[Bidirectional Whitham equation, (P1)]
The bidirectional Whitham equation \cite{MR3390078} is the system 
\begin{align*}
\begin{split}
    \eta_t &= -K*u_x -(\eta u)_x,\\
    u_t& = -\eta_x-uu_x,
  \end{split}
\end{align*}
with $K(x):=\mathcal{F}^{-1}\left(\frac{\tanh(\xi)}{\xi}\right)$. Equivalently,
\begin{align*}
	\vec{v}_t(t,x)=\vec{F}(D, \vec{v})(t,x),
\end{align*}
where $\vec{v}=(\eta, u)$ and $\vec{F}(D, \vec{v})=(-K*u_x -(\eta u)_x, -\eta_x-uu_x)$. With $\F(\hat{\vec v}(t,\xi), \xi) = \mathcal{F}(D,F(\vec v))(t,\xi)$ one has that
\begin{align*}
    \F(\hat{\vec v}(t,\xi), \xi)= \left(-\frac{\tanh(\xi)}{\xi} (\i\xi)\hat{u}(t,\xi) -(\i\xi)[\hat{\eta}*\hat{u}](t,\xi), -(\i\xi)\hat{\eta}-[\hat{u}*(\i\cdot)\hat{u}](t,\xi)\right),
\end{align*}
and therefore
\begin{align*} 
    \F(\hat{\vec v}(t,\xi), -\xi) =-\F(\hat{\vec v}(t,\xi), \xi),
\end{align*}
since $\tanh$ is odd. Hence, any  spatially symmetric and unique classical solution $(\eta, u)$ of the bidirectional Whitham equation is a travelling wave.
\end{Example}

\begin{Example}[Nonlocal fractional Keller--Segel system, (P2)]
The nonlocal fractional Keller--Segel system of chemotaxis in $I \times \R^{n}$, $n\geq 2$, describes directed movement of cells in response to the gradient of a chemical \cite{Z10}. In two space dimensions, it is given by
\begin{equation}\label{KS}
\begin{aligned}
&\partial_t u+ (-\Delta)^{\frac{\theta}{2}}u=-\nabla \cdot (uB(v)),\\
&\partial_t v+(-\Delta)^{\frac{\theta}{2}}v=u, 
\end{aligned}
\end{equation}
where  $\theta \in (1,2]$, $B(u)=\nabla\left((-\Delta)^{\frac{-\theta_1}{2}}\right)$ and $\theta_1 \in [0,2)$. Since the symbol of the fractional Laplacian $(-\Delta)^{\frac{s}{2}}$ is given by $|\xi|^s$, it is clear that writing \eqref{KS} as
\[
	\partial_t \vec{w} =\vec{F}(\vec w),
\]
with $\vec w = (u,v)$, both components of the function 
\[
\vec{F}(\vec w)= (-(-\Delta)^{\frac{\theta}{2}}u-\nabla \cdot (uB(v)),-(-\Delta)^{\frac{\theta}{2}}v-u)
\]
are even in the sense of \eqref{eq:higher-dim Fourier parity}. Hence, if $\vec w$ is a classical and unique solution for which the initial data are spatially symmetric with respect to an axis of symmetry \(\boldsymbol \lambda \in \R^2\), then the solution is symmetric for all (forward) times with respect to the same, fixed \(\boldsymbol \lambda\).
\end{Example}

\subsection{Euler equations}
\label{s-euler}
Our last example is the incompressible Euler equations posed in a two-dimensional domain contained in physical vacuum \cite{MR2291920},
\begin{align}
\label{e-wwp}
\begin{split}
    &u_t + uu_x+vu_y =  -P_x,\\ 
    &v_t + uv_x+vv_y =  -P_y,\\
    &u_x+v_y=0,\\
    &v=\eta_t+\eta_x u, \; \;P=0 \quad \text{on}\quad y=\eta(t,x),\\
    &v=0 \quad \text{on}\quad  y=-d. 
    \end{split}
\end{align}
Here, $(u,v)$ is the velocity field, not necessarily irrotational, $P$ is the pressure, and $\eta$ is the graph of the free surface. We denote the fluid domain by $\Omega(t)=\{(x,y) \mid  x\in \R, -d \leq y \leq \eta(t,x) \}$, and call a solution $(u,v,\eta,P)$ horizontally symmetric if its stream function (defined by \(\nabla \psi = (-v,u)\)) is symmetric in \(x\). Written out in \((u,v,\eta)\), this means that
\begin{align}
\label{def_symmetry}
\begin{split}
    u(t,x,y)&=u(t,2\la(t)-x,y),\\
    v(t,x,y)&=-v(t,2\la(t)-x,y),\\
    \eta(t,x)&=\eta(t,2\la(t)-x),
   \end{split}
\end{align}
for some function $\la \in C^1(\R)$, and all $t>0$. Let \(n\) denote the outward unit normal at the free surface. We shall assume that the initial datum satisfies the Rayleigh--Taylor condition
  \begin{equation}
  \label{e-RTcond}
      \nabla P \cdot n <0\qquad \text{ on }\quad y=\eta(t_0,x),
   \end{equation}
 and that the solution is classical in the sense that  $u,v, P \in C^1(I,C^2(\R))$ and $\eta \in C^1(I\times \R)$. To ensure uniqueness, we shall in our example use the following well-posedness result.

\begin{lemma}\cite[Thm 1.2]{MR2291920}
\label{lemma-WPgravity}
Assume that $u_0, v_0, \eta_0 \in H^3(\Omega(t_0))$ and that the condition
  \begin{equation}
  \label{e-RTcond}
      \nabla P \cdot n <0 \text{ on } y=\eta(t_0,x)
   \end{equation}
holds initially. Then, there exists $T> 0$ and a unique solution $(u(t), v(t), P(t),\eta(t))$ of  \eqref{e-wwp} with $u,v \in L^{\infty}(0,T;H^3(\Omega))$, $P\in  L^{\infty}(0,T;H^{\frac{7}{2}}(\Omega))$ and $\eta \in H^3(\R)$. 
\end{lemma}

For irrotational waves in the setting of infinite depth \cite{Wu97}, or in the setting of a flat bottom \cite{Lannes2005}, the condition \eqref{e-RTcond} is automatically satisfied. Our result is:

\begin{proposition}
\label{prop-EulerSymSteady}
Any horizontally symmetric solution of \eqref{e-wwp} whose initial datum satisfies the assumptions of Lemma~\ref{lemma-WPgravity} constitutes a travelling solution. 
\end{proposition}

\begin{remark}
The first corresponding result for the Euler equations can be found in~\cite{Ehrnstrom2009a}, assuming symmetry also for \(P\) and in the presence of gravity.
In the presence of constant vorticity, it was shown in \cite{K16} that symmetry of the wave profile and the horizontal velocity component on the surface is enough to guarantee that spatially periodic solutions are travelling waves in that case. In the irrotational setting, a corresponding assumption at the flat bed is enough to guarantee that periodic symmetric waves are travelling \cite{K152}. Our result is different in that it does not assume anything on the curl of the flow, but requires instead the setting of \eqref{def_symmetry} and Lemma~\ref{lemma-WPgravity}.
\end{remark}

\begin{remark}
Note that standing waves, which exist for the Euler equations \cite{MR2187619}, do not violate the above results, since such waves do not fulfil the corresponding assumptions.
\end{remark}

\begin{proof}[Proof of Proposition~\ref{prop-EulerSymSteady}]
 Let  $(u,v,\eta,P)$ be a horizontally symmetric solution of \eqref{e-wwp} whose initial datum is such that the solution is unique. Taking the curl of the first and second equation in \eqref{e-wwp} yields
\begin{equation}
\label{e-curl}
    \Big(u_{ty} +(uu_x)_y+(vu_y)_y\Big)-  \Big(v_{tx} +(uv_x)_x+(vv_y)_x\Big)=0,
\end{equation}
in view of that $-P_{xy}+P_{yx}=0$.  Evaluating \eqref{e-curl} at $(t,2\la(t)-x,y)$ and taking the symmetry assumption \eqref{def_symmetry}  into account, we return to original variables $(t,x,y)$ and obtain 
\begin{equation*}
    \Big(u_{ty} +2\dot \la u_{xy} -(uu_x)_y-(vu_y)_y\Big)
    - \Big(v_{tx} +2\dot\la v_{xx}-(uv_x)_x-(vv_y)_x\Big)=0.
\end{equation*}
Subtracting this from \eqref{e-curl} yields
\begin{equation}
\label{e-curldiff}
    \Big((u-\dot \la) u_{x}  + vu_y)\Big)_y
    - \Big((u-\dot\la) v_{x}+vv_y\Big)_x=0,
\end{equation}
and the same procedure for the kinematic boundary condition at the free surface gives
\begin{equation}
\label{e-BC1diff}
    v=(u-\dot \la)\eta_x\quad \text{ on } y=\eta(t,x).
\end{equation}
Now fix a time $t_0$, let $c:=\dot\la(t_0)$, and define
\begin{align*}
    \ub(x,y) &= u(t_0,x,y),\\
    \vb(x,y) &= v(t_0,x,y),\\
    \eb(x) &= \eta(t_0,x).
\end{align*}
By construction, $(\ub,\vb,\eb)$ satisfies \eqref{e-curldiff} and \eqref{e-BC1diff}, that is,
\begin{align}
\label{e-curlfree}
  \Big((\ub-c) \ub_{x}  + \vb\ub_y)\Big)_y
    - \Big((\ub-c) \vb_{x}+\vb\vb_y\Big)_x=0,
\end{align}
and
\begin{align*}
   \vb = (\ub-c)\eb_x \quad \text{on } y=\eb(x).  
 \end{align*}
Due to \eqref{e-curlfree}, there exists a function $F=F(x,y)$ such that
 \begin{align*}
    &(\ub-c) \ub_{x}  + \vb\ub_y= -F_x.\\
    &(\ub-c) \vb_{x}+\vb\vb_y= -F_y,
 \end{align*}
 where $F$ is defined uniquely up to an additive constant. Next we define 
\begin{align*}
    \ut(t,x,y) &= \ub(x-c(t-t_0),y),\\
    \vt(t,x,y) &= \vb(x-c(t-t_0),y),\\
    \et(t,x) &= \eb(x-c(t-t_0)),\\
    \tilde F(t,x,y) &= F(x-c(t-t_0),y).
\end{align*}
 These functions satisfy 
 \begin{align*}
    &\ut_t +\ut\ut_x + \vt\ut_y= -\tilde F_x,\\
    &\vt_t+\ut\vt_x + \vt\vt_y= -\tilde F_y,\\
    &\ut_x+\vt_y = 0,\\
    & \vt=\et_t+\ut\et_x, \quad \text{on}\quad y=\et(t,x),\\
    & \vt = 0 \quad \text{on } y=-d.
 \end{align*} 
 In order to verify that $\tilde F=0$ on $y=\et(t,x)$, we compute
 \begin{align*}
 -\nabla \tilde F (t,x, \tilde \eta(t,x))\cdot (1, \tilde \eta_x(t,x))&= \left( \tilde u_t +\tilde u \tilde u_x + \tilde v \tilde u_y\right)(t,x, \tilde \eta(t,x))\\
 &\quad + \tilde \eta_x(t,x) \left(\tilde v_t + \tilde u \tilde v_x + \tilde v \tilde v_y \right)(t,x, \tilde \eta(t,x))\\
 &=\left( -c \bar u_x +\bar u \bar u_x + \bar v \bar u_y\right)(x, \bar \eta(x-c(t-t_0)))\\
 &\quad + \bar \eta_x(t,x) \left(-c\bar v_x + \bar u \bar v_x + \bar v \bar v_y \right)(x, \bar \eta(x-c(t-t_0))).
 \end{align*}
 On the other hand, $\nabla P(t,x, \eta(t,x))\cdot (1, \eta_x(t,x))=0$, meaning
 \begin{align}\label{e-1}
u_t + uu_x +vu_y + \eta_x \left( v_t + uv_x+vv_y \right)=0
 \end{align}
 and
 \begin{align}\label{e-2}
 u_t + 2 \dot{\lambda}u_x - uu_x -vu_y + \eta_x \left( v_t+2\dot{\lambda}v_x - uv_x-vv_y \right)=0,
 \end{align}
 in view of the symmetry of $u, \eta$ and the antisymmetry of $v$.
 Subtracting \eqref{e-1} from \eqref{e-2} now yields that
 \begin{align*}
 -\dot{\lambda}u_x + uu_x+vu_y + \eta_x\left(-\dot{\lambda}v_x+uv_x+vv_y\right)=0.
 \end{align*}
From this we deduce that, at $t=t_0$, we have
 \begin{align*}
	  \nabla \tilde F (t_0,x, \tilde \eta(t_0,x))\cdot (1, \tilde \eta_x(t_0,x))=0.
 \end{align*}
Now, since 
\[
\nabla \tilde F (t,x, \tilde \eta(t,x))\cdot (1, \tilde \eta_x(t,x))=\nabla F(x-c(t-t_0), \bar \eta (x-c(t-t_0)))\cdot (1, \bar \eta_x(x-c(t-t_0))), 
\]
we obtain in fact that $\tilde{F}=0$ all along $y=\tilde \eta(t,x)$. To summarize, we have shown that $(\tilde u, \tilde v, \tilde \eta, \tilde F)$ is a solution to \eqref{e-wwp}. 

To deduce uniqueness from Lemma~\ref{lemma-WPgravity}, we are left to show that $\tilde F$ satisfies the Rayleigh-Taylor condition at $t=t_0$. As above, we take advantage of the information about $\nabla P \cdot n$. Since $P$ satisfies the Rayleigh-Taylor condition, we have $-\nabla P (t_0,x,\eta(t_0,x))\cdot(-\eta_x(t_0,x),1)>0$.  In particular, the symmetry relations imply that
 \begin{align} \label{e-pressure1}
 \left(-u_t-uu_x-vu_y\right)\eta_x+v_t+uv_x+vv_y > 0, 
 \end{align}
 and
 \begin{align} \label{e-pressure2}
  \left(u_t+2\dot{\lambda}u_x-uu_x-vu_y\right)\eta_x-v_t-2\dot{\lambda}v_x+uv_x+vv_y > 0.
  \end{align}
 Adding \eqref{e-pressure1} to \eqref{e-pressure2}, we get
 \begin{align}\label{e-pressure}
 \left(\dot{\lambda}u_x-uu_x-vu_y\right)\eta_x - \dot{\lambda}v_x+uv_x+vv_y>0.
 \end{align}
Finally, computing the derivative of $-\tilde F$ in the outward normal direction, we arrive at
 \begin{align*}
 -\nabla \tilde F (x,\tilde \eta)(-\eta_x,1)=\left(c\tilde u_x-\tilde u \tilde u_x -\tilde v \tilde u_y\right)\tilde \eta_x -c\tilde v_x + \tilde u \tilde v_x +\tilde v\tilde v_y,
 \end{align*}
 which coincides with \eqref{e-pressure} at $t=t_0$ and, therefore,
 \[
 \nabla \tilde F\cdot n < 0 \qquad \mbox{on}\quad y=\tilde \eta(t_0,x).
 \]
 In view of that $(\tilde u, \tilde v, \tilde \eta)= (u,v, \eta)$ at initial time $t=t_0$, Lemma~\ref{lemma-WPgravity} allows us to conclude that $(u,v,\eta,P)$ is a travelling wave.
 \end{proof}

\bibliographystyle{siam}
\bibliography{BGEP-symmetric,Library_all}   

\begin{thebibliography}{10}

\bibitem{Alk}
{\sc A.~D. Aleksandrov}, {\em Uniqueness theorems for surfaces in the large.
  {V}}, Amer. Math. Soc. Transl. (2), 21 (1962), pp.~412--416.

\bibitem{A09}
{\sc J.~Angulo~Pava}, {\em Nonlinear dispersive equations}, vol.~156, American
  Mathematical Society, Providence, RI, 2009.

\bibitem{MR3603270}
{\sc G.~Bruell, M.~Ehrnstr\"om, and L.~Pei}, {\em Symmetry and decay of
  traveling wave solutions to the {W}hitham equation}, J. Differential
  Equations, 262 (2017), pp.~4232--4254.

\bibitem{NLS-breather}
{\sc O.~Chvartatskyi and F.~M\"{u}ller-Hoissen}, {\em {NLS} breathers, rogue
  waves, and solutions of the {L}yapunov equation for {J}ordan blocks}, J.
  Phys. A, 50 (2017), pp.~155--204.

\bibitem{ConEhrWah07}
{\sc A.~Constantin, M.~Ehrnstr{\"{o}}m, and E.~Wahl{\'{e}}n}, {\em {Symmetry of
  steady periodic gravity water waves with vorticity}}, Duke Math. J., 140
  (2007), pp.~591--603.

\bibitem{MR2291920}
{\sc D.~Coutand and S.~Shkoller}, {\em Well-posedness of the free-surface
  incompressible {E}uler equations with or without surface tension}, J. Amer.
  Math. Soc., 20 (2007), pp.~829--930.

\bibitem{Craig1988}
{\sc W.~Craig and P.~Sternberg}, {\em {Symmetry of solitary waves}}, Comm.
  Partial Differ. Equations, 13 (1988), pp.~603--633.

\bibitem{Ehrnstrom2009a}
{\sc M.~Ehrnstr{\"{o}}m, H.~Holden, and X.~Raynaud}, {\em Symmetric waves are
  traveling waves}, Int. Math. Res. Not., 2009 (2009), pp.~4578--4596.

\bibitem{GS07}
{\sc V.~A. Galaktionov and S.~R. Svirshchevskii}, {\em Exact solutions and
  invariant subspaces of nonlinear partial differential equations in mechanics
  and physics}, Chapman \& Hall/CRC, Boca Raton, FL, 2007.

\bibitem{Garabedian1965}
{\sc P.~Garabedian}, {\em {Surface waves of finite depth}}, J. d'Analyse Math.,
  14 (1965), pp.~161--169.

\bibitem{MR1749871}
{\sc M.~Garc{\'{\i}}a-Huidobro, R.~Man{\'a}sevich, J.~Serrin, M.~Tang, and
  C.~S. Yarur}, {\em Ground states and free boundary value problems for the
  {$n$}-{L}aplacian in {$n$} dimensional space}, J. Funct. Anal., 172 (2000),
  pp.~177--201.

\bibitem{GX09}
{\sc X.~Geng and B.~Xue}, {\em An extension of integrable peakon equations with
  cubic nonlinearity}, Nonlinearity, 22 (2009), pp.~1847--1856.

\bibitem{Geyer2016}
{\sc A.~Geyer}, {\em {Symmetric waves are traveling waves for a shallow water
  equation modeling surface waves of moderate amplitude}}, J. Nonlinear Math.
  Phys., 22 (2016), pp.~545--551.

\bibitem{GNN79}
{\sc B.~Gidas, W.~M. Ni, and L.~Nirenberg}, {\em Symmetry and related
  properties via the maximum principle}, Comm. Math. Phys., 68 (1979),
  pp.~209--243.

\bibitem{HKNS06}
{\sc N.~Hayashi, E.~I. Kaikina, P.~I. Naumkin, and I.~A. Shishmarev}, {\em
  Asymptotics for dissipative nonlinear equations}, vol.~1884, Springer-Verlag,
  Berlin, 2006.

\bibitem{HH}
{\sc J.~M. Hill and D.~L. Hill}, {\em High-order nonlinear evolution
  equations}, IMA J. Appl. Math., 45 (1990), pp.~243--265.

\bibitem{Hirota}
{\sc R.~Hirota}, {\em The direct method in soliton theory}, vol.~155 of
  Cambridge Tracts in Mathematics, Cambridge University Press, Cambridge, 2004.

\bibitem{MR2187619}
{\sc G.~Iooss, P.~I. Plotnikov, and J.~F. Toland}, {\em Standing waves on an
  infinitely deep perfect fluid under gravity}, Arch. Ration. Mech. Anal., 177
  (2005), pp.~367--478.

\bibitem{K152}
{\sc F.~Kogelbauer}, {\em On symmetric water waves with constant vorticity}, J.
  Nonlinear Math. Phys., 22 (2015), pp.~494--498.

\bibitem{K16}
\leavevmode\vrule height 2pt depth -1.6pt width 23pt, {\em On the symmetry of
  spatially periodic two-dimensional water waves}, Discrete Contin. Dyn. Syst.,
  36 (2016), pp.~7057--7061.

\bibitem{KL2017a}
{\sc V.~Kozlov and E.~Lokharu}, {\em Small-amplitude steady water waves with
  critical layers: non-symmetric waves}.
\newblock arXiv:1701.04991, 2017.

\bibitem{Lannes2005}
{\sc D.~Lannes}, {\em {Well-posedness of the water-waves equations}}, J. Am.
  Math. Soc., 18 (2005), pp.~605--654.

\bibitem{Lannes13a}
\leavevmode\vrule height 2pt depth -1.6pt width 23pt, {\em The water waves
  problem}, vol.~188, American Mathematical Society, Providence, RI, 2013.

\bibitem{Li91b}
{\sc C.~Li}, {\em Monotonicity and symmetry of solutions of fully nonlinear
  elliptic equations on bounded domains}, Comm. Partial Differential Equations,
  16 (1991), pp.~491--526.

\bibitem{Li91}
\leavevmode\vrule height 2pt depth -1.6pt width 23pt, {\em Monotonicity and
  symmetry of solutions of fully nonlinear elliptic equations on unbounded
  domains}, Comm. Partial Differential Equations, 16 (1991), pp.~585--615.

\bibitem{MR3390078}
{\sc D.~Moldabayev, H.~Kalisch, and D.~Dutykh}, {\em The {W}hitham {E}quation
  as a model for surface water waves}, Phys. D, 309 (2015), pp.~99--107.

\bibitem{O78}
{\sc L.~Ostrovsky}, {\em Nonlinear internal waves in a rotating ocean},
  Okeanologia, 18 (1978), pp.~181--191.

\bibitem{PS11}
{\sc P.~Popivanov and A.~Slavova}, {\em Nonlinear waves}, vol.~4, World
  Scientific Publishing Co. Pte. Ltd., Hackensack, NJ, 2011.
\newblock An introduction.

\bibitem{S71}
{\sc J.~Serrin}, {\em A symmetry problem in potential theory}, Arch. Rational
  Mech. Anal., 43 (1971), pp.~304--318.

\bibitem{SZJ07}
{\sc J.~Shu and J.~Zhang}, {\em Sharp condition of global existence for
  second-order derivative nonlinear {S}chr\"odinger equations in two space
  dimensions}, J. Math. Anal. Appl., 326 (2007), pp.~1001--1006.

\bibitem{TPDE3}
{\sc M.~E. Taylor}, {\em Partial {D}ifferential {E}quations {III}. {N}onlinear
  {E}quations}, vol.~117 of Applied Mathematical Sciences, Springer, New York,
  second~ed., 2011.

\bibitem{walsh2009}
{\sc S.~Walsh}, {\em {Some criteria for the symmetry of stratified water
  waves}}, Wave Motion, 46 (2009), pp.~350--362.

\bibitem{Wu97}
{\sc S.~Wu}, {\em {Well-posedness in Sobolev spaces of the full water wave
  problem in 2D.}}, Inven. Math., 130 (1997), pp.~39--72.

\bibitem{Wu01}
{\sc Z.~Wu, J.~Zhao, J.~Yin, and H.~Li}, {\em Nonlinear {D}iffusion
  {E}quations}, World Scientific Publishing Co., Inc., River Edge, NJ, 2001.

\bibitem{Z10}
{\sc Z.~Zhai}, {\em Global well-posedness for nonlocal fractional
  {K}eller-{S}egel systems in critical {B}esov spaces}, Nonlinear Anal., 72
  (2010), pp.~3173--3189.

\end{thebibliography}

\end{document}